\documentclass[10pt,a4paper]{article}
\usepackage[utf8]{inputenc}
\usepackage[english]{babel}
\usepackage{amsmath}
\usepackage{amsthm}
\usepackage{amsfonts}
\usepackage{amssymb}
\usepackage{graphicx}
\usepackage{mathrsfs}
\usepackage{enumerate}

\usepackage{hyperref}
\usepackage{xcolor}
\hypersetup{
  colorlinks=true,
  linkcolor={blue!80!black},
  citecolor=black,
  urlcolor={blue!80!black}
}

\theoremstyle{plain}
\newtheorem{theorem}{Theorem}[section]
\newtheorem{lemma}[theorem]{Lemma}
\newtheorem{corollary}[theorem]{Corollary}

\newtheorem{exercise}[theorem]{Exercise}
\newenvironment{solution}{\par \noindent {\it Solution.}}{\par \bigskip}

\theoremstyle{definition}
\newtheorem{definition}[theorem]{Definition}

\theoremstyle{remark}
\newtheorem{remark}[theorem]{Remark}
\newtheorem{example}[theorem]{Example}

\newcommand{\R}{{\mathbb R}}
\newcommand{\C}{{\mathbb C}}
\newcommand{\N}{{\mathbb N}}
\newcommand{\Z}{{\mathbb Z}}
\newcommand{\abs}[1]{{\left\lvert{#1}\right\rvert}}
\newcommand{\norm}[1]{{\left\lVert{#1}\right\rVert}}
\newcommand{\pv}{\operatorname{p.v.}}
\newcommand{\intlim}[2]{\bigg/_{\!\!\! #1}^{\, #2}}
\newcommand{\Ht}{{\mathscr H}}
\newcommand{\M}{\mathscr M}

\usepackage{environ}
\NewEnviron{killcontents}{}

\title{The Fourier, Hilbert and Mellin transforms on a half-line}

\date{July 25, 2023}

\usepackage{authblk}

\author[1,2]{Emilia L.K. Bl{\r a}sten}
\author[1]{Lassi P\"aiv\"arinta}
\author[1]{Sadia Sadique}

\affil[1]{{\small Division of Mathematics, Tallinn University of
    Technology, Department of Cybernetics, 19086 Tallinn, Estonia}}

\affil[2]{{\small Computational Engineering, School of Engineering
    Science, LUT University, 15210 Lahti, Finland}}

\begin{document}
\maketitle

\begin{abstract}
  We are interested in the singular behaviour at the origin of
  solutions to the equation $\Ht \rho = e$ on a half-axis, where $\Ht$
  is the one-sided Hilbert transform, $\rho$ an unknown solution and
  $e$ a known function. This is a simpler model problem on the path to
  understanding wave field singularities caused by curve-shaped
  scatterers in a planar domain.

  We prove that $\rho$ has a singularity of the form $\M[e](1/2) /
  \sqrt{t}$ where $\M$ is the Mellin transform. To do this we use
  specially built function spaces $\M'(a,b)$ by Zemanian, and these
  allow us to precisely investigate the relationship between the
  Mellin and Hilbert transforms. Fourier comes into play in the sense
  that the Mellin transform is simpy the Fourier transform on the
  locally compact Abelian multiplicative group of the half-line, and
  as a more familiar operator it guides our investigation.
\end{abstract}

\section{Introduction}
In the present article we let R.H. Mellin meet J.B.J. Fourier and
D. Hilbert. More exactly we study the connection of the Mellin
transform to the Hilbert- and Fourier transforms in a half-axis
$\R_+=(0,\infty)$. Mellin defined his transform in 1886
\cite{mellin87_uber_einen_zusam_zwisc_gewis} in connection with his
studies on certain difference- and differential equations. A bit more
than a decade later Hilbert presented a new singular integral
transform \cite{hilbert05_uber_anwen_integ_probl_funkt} in the third
International Congress of Mathematicians, 1904, where he gave a
lecture about the Riemann--Hilbert problem. Fourier's work preceded
these works of Mellin and Hilbert by more than 60 years
\cite{fourier22_theor}.

The classical Hilbert transform on the real line is defined by the
formula
\begin{equation}\label{eq:classical-hibert}
  \Ht f(x) = \pv \int_{-\infty}^\infty \frac{f(y)}{x-y} dy.
\end{equation}
The connection to the Fourier transform $\mathscr F$ is the well-known
formula
\begin{equation}\label{eq:intro-hilbert-fourier}
  \mathscr F(\Ht f)(\xi) = i \operatorname{sgn}\xi \widehat f(\xi)
\end{equation}
where $\widehat f=\mathscr F f$, see
\cite{king09_hilber_trans_vol1,stein70_singul,stein93_harmon}. However,
the Mellin transform is defined on a half-axis and the connection to
the Hilbert transform and especially to the Fourier transform is less
widely known, despite being a quite old results
\cite{folland92_fourier,havin98_commut_harmon_analy_ii}. The secret to
these connections is lying on the fact that the half-axis is a locally
compact Abelian group with respect to multiplication. The Fourier
transform is well-defined in all such groups and the convolution
theorem holds \cite{rudin90_fourier_analy_group}. Since the one-sided
Hilbert transform \cite{king09_hilber_trans_vol1} satisfies
\[
\Ht f(t) = \pv \int_0^\infty \frac{f(t/s)}{1-s} ds
\]
which is a convolution in the multiplicative group $(\R_+,\cdot)$, we
have discovered the connection of the Hilbert- and Fourier
transforms\footnote{This is why we study the Hilbert transform on a
half-axis and not on a finite interval as in Section~4 of
\cite{tricomi57_integ_equat} or in
\cite{astala96_finit_hilber_trans_weigh_spaces}.} in $\R_+$. It
remains to find out the Foutier transform in $\R_+$. After this
lengthy introduction is should be no big surprise that it is exactly
the Mellin transform. All of this is explained with more detail in
Section~\ref{sec:hilb-mell-transf} below.

\medskip
In this article we are interested in the so-called one-sided Hilbert
transform
\begin{equation}
  {\Ht}f(x) = \pv \frac1\pi \int_{0}^{+\infty} \frac{f(y)}{x-y} dy.
\end{equation}
Other terminology for this transform are the reduced Hilbert
transform, the half-Hilbert transform or the semi-infinite Hilbert
transform \cite[Section 12.7]{king09_hilber_trans_vol1}. Our interest
is in understanding the existence, uniqueness and behaviour at the
origin of solutions $\rho$ to the inhomogeneous equations
\begin{equation}\label{eq:inhomogeneous-intro}
  \Ht\rho = e
\end{equation}
for a given $e$.

Equation (\ref{eq:inhomogeneous-intro}) has previously been studied in
classical context, with $\rho$ and $e$ being classicaly smooth or
Lebesgue integrable. See for example
\cite{bonis02_approx_hilber_trans_real_semiax,olver11_comput_hilber_trans_its_inver,paveri-fontana94_half_hartl_half_hilber_trans,paveri-fontana94_errat}. These
references have a practical point of view, with emphasis on
computations or asymptotic expantions.

Our motivation is to understand the singular behaviour of the solution
in cases where the right-hand side might not be smooth or integrable
in the classical sense. The motivation for this comes from studying
scattering of quantum or acoustic waves from a crack or screen in a
two-dimensional domain. The three-dimensional problem for a flat
two-dimensional scattering screen was studied in
\cite{blasten20_unique_deter_shape_scatt_screen}. In that paper, an
incident probing wave $u_i$ satisfying $(\Delta+k^2)u_i=0$ in $\R^3$
reacts with a screen $S$ and as a concequence a scattered wave $u_s$
is emitted. These are tied together mathematically as follows:
\begin{align}
  (\Delta +k^2) u_s = 0,& \qquad
  \R^3\setminus\overline S,\\ u_i(x) + u_s(x) = 0,& \qquad
  x\in S,\\ r \Big(\frac{\partial}{\partial r} - ik\Big)u_s = 0,&
  \qquad r\to\infty,
\end{align}
where $r=\abs{x}$ and the limit is uniform over all directions $\hat x
= x/r$ as $r\to\infty$. The research question was whether the
\emph{far-field pattern} of $u_s$ uniquely determines the shape $S$.
Analysing the problem lead to studying the support of a generalized
function $\rho$ which satisfies an integral equation of the form
\begin{equation} \label{eq:3D-integral-eq}
  - \int_{S} \Phi(x-y) \rho(y) d\sigma(y) = u_i(x),
\end{equation}
where $\Phi$ is the Green's function for $\Delta+k^2$ in three
dimensions. Notice how it is analogous to
(\ref{eq:inhomogeneous-intro}).

The methods in \cite{blasten20_unique_deter_shape_scatt_screen} apply
to flat scatterers. For more general objects it is fruitful to study
the singular behaviour of solutions to inhomogeneous integral
equations as above, see
\cite{alessandrini13_crack_with_imped,friedman89_deter_crack_by_bound_measur,stephan84_augmen_galer_proced_bound_integ}
and the references therein related to the crack problem for the
conductivity equation. The problem has yet to be solved in the
acoustic setting.

This study is our first step into understanding the singular behaviour
of waves near the endpoint of cracks or screens in an acoustic
medium. Simplifying the applied problem leads to the study of the
equation $\Ht\rho=e$ on the half-line in a class of generalized
functions. Our approach is to use the Mellin transform
\begin{equation}\label{eq:Mellin-intro}
  \M[f](s) = \int_0^\infty f(t) t^{s-1} dt
\end{equation}
defined for generalized functions. We follow the approach of Zemanian
\cite{zemanian69_gener_integ_trans}. See sections
\ref{sec:space-mathscrm-mell} and \ref{sec:mell-transf-distr} for more
details. We then see how the Hilbert transform applies to these
generalized functions in Section~\ref{sec:hilbert-transform}. In
Section~\ref{sec:inhom-hilb-transf} we prove the following
theorems. But first some explanation of the notation.

An intuitive way of thinking of these spaces is that $u\in\M'(a,b)$ if
informally
\begin{align*}
  u(t) &= O(t^{-a}), \qquad t\to0,\\ u(t) &= O(t^{-b}), \qquad
  t\to\infty.
\end{align*}
A more precise understanding is that $u\in\M'(a,b)$ if the Mellin
transform $\M[u](s)$ is holomorphic in the vertical strip $s\in
S(a,b)$ defined by $a<\Re(s)<b$ and has polynomial growth on vertical
lines. This is enough to understand our theorems.

\begin{theorem}\label{thm:no-singularity}
  Let $e\in\M'(a,b)$ with $0\leq a<b\leq1$. If $b\leq1/2$ or $1/2\leq
  a$ or $a<1/2<b$ and $\M[e](1/2)=0$ the equation
  \[
  \Ht\rho = e
  \]
  has a unique solution $\rho=\rho_0\in\M'(a,b)$. Furthermore if
  $\rho'\in\M'(a',b')$ is another solution with $S(a',b')\subset
  S(a,b)$ then $\rho'=\rho_0$ in $\M'(a',b')$.
\end{theorem}

\begin{theorem}\label{thm:with-singularity}
  Let $e\in\M'(a,b)$ with $0\leq a <1/2 < b \leq 1$ and
  $\M[e](1/2)\neq0$. Then $\Ht\rho=e$ has no solutions $\rho$ whose
  Mellin transform contains $s=1/2$ in its strip of
  holomorphicity. Instead there are unique solutions
  $\rho_-\in\M'(a,1/2)$ and $\rho_+\in\M'(1/2,b)$ and they satisfy
  \begin{equation}\label{eq:solution-difference}
    \rho_+(t) - \rho_-(t) = \frac4\pi\M[e](1/2) \frac1{\sqrt{t}}.
  \end{equation}
  Furthermore if $\rho'\in\M'(a',b')$ is another solution with
  $S(a',b')$ intersecting $S(a,1/2)$ or $S(1/2,b)$ then $\rho'=\rho_-$
  or $\rho'=\rho_+$ in $\M'(a',b')$, respectively.
\end{theorem}

The Equation~(\ref{eq:solution-difference}) shows that $\rho_+$ has a
singularity of type $t^{-1/2}$ unless the Mellin transform of $e$
vanishes at $s=1/2$. This suggests that acoustically scattered waves
from most cracks or screens will have a singularity at their
ends. However, if
\[
e(t) =
\begin{cases}
  e^{i\sqrt{t}}, &0\leq t\leq(2\pi)^2,\\ 0, &t>(2\pi)^2,
\end{cases}
\]
it turns out that $\M[e](1/2)=0$. In this case some incident plane
wave might not have as strong a singularity at $t=0$ for the curve
$\Gamma(t)=(t,\sqrt{t})$ as for most other curves or incident waves.
Further analysis is needed and will appear in forthcoming papers, but
on this paper we focus on the intrinsic properties of the one-sided
Hilbert transform.

One might wonder what is the role of the point $s=1/2$ in the theorems
above. It arises as the only zero of the Mellin transform
$\cot(\pi s)$ of the kernel of the Hilbert transform $\Ht$ that's in
the strip $0<\Re s<1$. This strip comes from the technical proof
showing that the kernel $\pv 1/(1-t)$ is Mellin-transformable, see
Lemma~\ref{lem:pv-Mellin-transformable}.

\section{Hilbert- and Mellin transforms for measurable functions}
\label{sec:hilb-mell-transf}
In this section we define the Hilbert transform and Mellin transform
in $\R_+$ and establish their connection. Before that we recall some
known facts about Fourier transforms on locally compact abelian
groups. Then we show that in the case of the multiplicative group
$(\R_+,\cdot)$ we get exactly the Hilbert transform.

\subsection*{Definition of the LCA and Haar measure}
Let $G=(X,\cdot)$ be any locally compact Abelian group (LCA). Usually
\cite{rudin90_fourier_analy_group} the group operation is denoted by
addition and identity element by $0$. Since our main interest is the
multiplicative group $G_+ = (\R_+,\cdot)$ we denote the group
operation by a product $xy$, $x,y\in X$ and by $1$ the identity
element.

It is well known that there exists a measure $m$ on $X$ that is
invariant in the group action i.e.
\begin{equation}\label{eq:invariant-group-action}
  m(xE) = m(E)
\end{equation}
for every $x\in X$ and every Borell set $E$. Such a measure is called
the \emph{Haar measure} and it is unique up to a positive constant. If
$m$ and $m'$ are two Haar measures on $G$ then $m'=\lambda m$ for some
$\lambda>0$. It is quite easy to see that in $G_+=(\R_+,\cdot)$ the
Haar measure is $dt/t$ i.e. the measure $m$ with
\begin{equation}\label{eq:haar-Rplus-measure}
  m(E) = \int_E \frac{dt}t
\end{equation}
for any Borell set in $\R_+$.

If $m$ is a Haar measure on a LCA group $G$ we write $L^p(G)$ instead
of $L^p(m)$. Note that
\begin{equation}\label{eq:LpG-norm}
  \norm{f}_{L^p(G)} = \left( \int_X \abs{f(x)}^p dm(x) \right)^{1/p}
\end{equation}
is scaling invariant: if $f_x(y) = f(yx^{-1})$ then
$\norm{f_x}_{L^p(G)} = \norm{f}_{L^p(G)}$. In particular for $G_+$ we
have $f_t(s)=f(s/t)$ and
\begin{equation}\label{eq:Rplus-Lp-invariance}
  \int_{\R_+}\abs{f_t(s)}^p \frac{ds}{s} = \int_{\R_+} \abs{f(s)}^p
  \frac{ds}{s}
\end{equation}
which can of course be obtained also directly by changing variables.

\subsection*{Fourier transforms in a LCA}
If $G=(X,\cdot)$ is a LCA we call a function $\gamma:X\to\C$ a
\emph{character}, if $\abs{\gamma(x)}=1$ for all $x\in X$ and
\begin{equation}\label{eq:character}
  \gamma(x\cdot y) = \gamma(x)\gamma(y)
\end{equation}
for every $x,y\in X$. So a character on $G$ is a homomorphism from $G$
to $T$ where $T$ is the group of rotations of the unit circle in the
complex plane.

The set of all characters on a given LCA is denoted by $\Gamma$. We
equip it with multiplication
\begin{equation}\label{eq:character-multiplication}
  (\gamma_1\gamma_2)(x) = \gamma_1(x)\gamma_2(x)
\end{equation}
for $x\in X$. This makes $\Gamma$ a group. It is called the \emph{dual
group of $G$}.

We are ready to define the Fourier transform of $f\in L^1(G)$ by
\begin{equation}\label{eq:G-Fourier}
  \widehat f(\gamma) = \int_X f(x) \gamma(x^{-1}) dm(x)
\end{equation}
for $\gamma\in\Gamma$. We denote
\begin{equation}\label{eq:character-notation}
  \gamma(x) = (x,\gamma)
\end{equation}
from now on.

\begin{example}
  \begin{enumerate}
  \item If $G=(\R,+)$ we have for $\xi\in\R$ that
    \[
    \gamma_\xi(x) = e^{ix\xi}
    \]
    is a character and by denoting $\gamma_\xi$ simply by $\xi$, the
    Fourier transform turns out to be
    \[
    \widehat f(\xi) = \int_{-\infty}^\infty f(x) e^{-ix\xi} dx.
    \]
    Hence the dual group of $(\R,+)$ is $(\R,+)$ itself.
  \item If $G=T$, the dual group is $(\Z,+)$ and
    \[
    \widehat f(n) = \frac{1}{2\pi} \int_0^{2\pi} f(e^{i\theta})
    e^{-in\theta} d\theta.
    \]
  \item By Pontryagin Duality Theorem the dual group of $\Z$ is $T$
    and
    \[
      \widehat f(e^{ix}) = \int_{-\infty}^\infty f(n) e^{-inx}dm_\Z(n)
      = \sum_{n=-\infty}^\infty f(n) e^{-inx}.
    \]
  \end{enumerate}
\end{example}

The \emph{convolution} of $f\in L^1(G)$ and $g\in L^p(G)$, $1\leq
p<\infty$ is defined as
\begin{equation}\label{eq:lca-convolution}
  {f\ast g}(x) = \int_X f(xy^{-1}) g(y) dm(y)
\end{equation}
and the \emph{convolution theorem}
\begin{equation}\label{eq:lca-conv-transform}
  \widehat{f\ast g}(\gamma) = \widehat f(\gamma) \widehat g(\gamma)
\end{equation}
holds in any LCA \cite{rudin90_fourier_analy_group}.

To find out the Fourier transform in the group of our main interest,
$G_+=(\R_+,\cdot)$, we need to find its dual space $\Gamma$. But this
is simple: For $z=ix$, $x\in\R$, define
\begin{equation}\label{eq:Rplus-homomorphism}
  \gamma_z(t) = t^z = t^{ix}, \qquad t\in\R_+.
\end{equation}
Clearly this is a character in $G_+$ since
\[
\gamma_z(ts) = (ts)^{ix} = t^{ix}s^{ix}
\]
for $s,t\in\R_+$.

It is not difficult to see (\cite{rudin90_fourier_analy_group}
Section~2.2) that there are no other characters. Hence we can
interpret that the dual group of $G_+$ is the additive imaginary axis
of the complex plane and the Fourier transform is given by
\begin{equation}\label{eq:Gplus-Fourier}
  \widehat f(z) = \int_0^\infty t^z f(t) \frac{dt}t
\end{equation}
for $f\in L^1(G_+)$ and $z\in i\R$.

But this is exactly the definition of the Mellin transform
\cite{mellin87_uber_einen_zusam_zwisc_gewis,titchmarsh48_introd_theor_fourier_integ}
whenever the right-hand side is integrable. Thus we have shown that
the Mellin transform is nothing else than the Fourier transform in the
multiplicative group on $\R_+$. Accordingly, all the results for the
Fourier transforms in LCA's, such as Plancherel's theorem, the
inversion formula and convolution theorem follow now, as a matter of
routine, from the general theory of Fourier analysis in LCA's
\cite{rudin90_fourier_analy_group}. The connection to Hilbert
transform is in the formula
\begin{equation}\label{eq:intro-Hilbert-convolution}
  \Ht f(t) = \pv \int_0^\infty \frac{1}{1-t/s}f(s) \frac{ds}s = h\vee
  f (t)
\end{equation}
where $h = \pv \frac{1}{1-t}$ and $\vee$ stands for the Mellin
convolution is $(\R_+,\cdot)$. The convolution theorem suggests that
(\ref{eq:intro-Hilbert-convolution}) implies that the Mellin transform
of $\Ht f$ is
\begin{equation}\label{eq:intro-Mellin-Hilbert}
  \M\Ht f(z) = \widehat h(z) \widehat f(z) = \cot(\pi z) \widehat f(z)
\end{equation}
where $\widehat{\phantom{f}}$ is the Fourier transform on the LCA
$(\R_+,\cdot)$, or in other words, the Mellin transform. The second
equality follows from Example 8.24.II in
\cite{pap99_compl_analy_examp_exerc},
\begin{equation}\label{eq:intro-Mellin-Hilbert-kernel}
  \pv \int_0^\infty t^z \frac{1}{1-t} \frac{dt}t = \pi \cot(\pi z).
\end{equation}
The problem is that $h$ is not a function but a proper
distribution. The theory of distributions does not exist for general
LCA's and we must develop the theory for Mellin and Hilbert transforms
specifically for the group $(\R_+,\cdot)$. This is done in the
sections below.

\subsection*{Implications of LCA theory}
To the end of this introduction we give an exercise on how to use this
new connection of the Fourier transform in LCA and the Mellin
transform to prove generally challenging results. For the reader's
convenience we also give its solution.

\begin{exercise}
  Assume that $f\in L^1(\R_+, dt/t)$ and that its Mellin transform
  $\M f\in L^1(i\R)$. Then $f$ must be continuous and
  \begin{equation}\label{eq:limit-at-zero}
    \lim_{t\to0+} f(t)=0.
  \end{equation}
\end{exercise}
Before giving a solution we make two remarks about the result. It is
relatively easy to construct a function in $L^1(\R_+,dt/t)$ which is
continuous but the limit in (\ref{eq:limit-at-zero}) does not
exist. We can even construct it so that it is positive and
unbounded. However, if the limit exists then it must be equal to zero.

\begin{solution}
  We denote $G_+=(\R_+,dt/t)$ and by $\Gamma_+$ its dual group
  $(i\R,+)$. For any locally compact Abelian group $G$ the Fourier
  transform $\widehat f$ of a function belonging to $L^1(G,m)$, $m$
  being a Haar measure, is in the space $C_0(\Gamma)$ where $\Gamma$
  is its dual group and $C_0(\Gamma)$ is the closure of compactly
  supported continuous functions in $L^\infty(\Gamma)$
  \cite[Section~1.2.3]{rudin90_fourier_analy_group}. Hence in our case
  $\widehat f\in L^1(i\R)\cap C_0(i\R)$. We don't need this to solve
  the exercise but use instead \emph{Pontryagin's duality theorem}
  \cite[Section~1.5]{rudin90_fourier_analy_group} to get first
  $f(t)=\mathscr{F}g(-t)$ where $g$ is the Fourier transform of $f$,
  namely $g=\widehat f$. Next, we apply the above result in the context of the dual 
  pair $(\Gamma_+,G_+)$ instead of the original pair $(G_+,\Gamma_+)$. We finally
  obtain that $f\in C_0(G_+)$ which means that $f$ is continuous and
  $f(t)=0$ when $t\to0$.
\end{solution}

\section{Space of Mellin transformable
  distributions} \label{sec:space-mathscrm-mell} In this section we
define a class of distributions on the positive real axis. The Mellin
transform of these distributions will be functions that are
holomorphic on a vertical strip in the complex plane and also
polynomially bounded as the imaginary part of the argument grows. This
class of distributions will be denoted by $\M'(a_1,a_2)$ where
$a_1, a_2 \in \R$ define the strip of holomorphicity. The construction
is analogous to how tempered distributions $\mathscr{S}'(\R) $ are
defined for extending the range of the Fourier transformation.

The strategy is loosely described in \cite{bertrand10_mellin_trans}
which follows \cite{misra86_trans}. The general idea is to define
spaces of ordinary smooth test functions on $\R_+$ which contain
compactly supported smooth test functions $\mathscr{D}(\R_+)$ and also
functions of the form $t^{s-1}$ for some complex numbers $s$. One then
defines the duals of these as the spaces of interest. We note that
both \cite{bertrand10_mellin_trans} and \cite{misra86_trans} are scant
on the precise details. In fact the latter uses the notation
$\mathscr{T}_{p,q}$ and implicitly $\mathscr{T}_{\alpha,\omega}$ to
mean different things. This causes confusion when applied to real
cases. For example the function $g(t)=1$ for $0<t<1$ and $g(t)=0$ for
$t\geq1$ belongs to $\mathscr{T}_{0,1}$ when interpreted in the latter
way but not in the former. A more reliable reference is
\cite{zemanian69_gener_integ_trans}. Although the test function spaces
are defined differently than in the former references, the final space
of Mellin transformable distributions ends up being the same.

Section 11.3.3. in \cite{misra86_trans} compares their initial test
function space $\M_{p,q}$ to spaces $\M(a,b)$ defined by Zemanian in
\cite{zemanian69_gener_integ_trans} and concludes rightly that the
function $g$ above does not belong to $\M'_{0,\infty}$. However these
are not defined in Zemanian; instead a larger space $\M'(0,\infty)$ is
defined and it does contain that function.

\bigskip
We start by describing a space of test functions which will be used to
define the Mellin transform of a class of distributions. This
summarises Section~4.2 of Zemanian
\cite{zemanian69_gener_integ_trans}.

\begin{definition} \label{def:test-functions}
  Let $a_1 < a_2$ be real numbers. Then $\M_{a_1,a_2}$ contains all
  smooth functions $\phi: \R_+ \rightarrow \C $ such that for any
  $k\in\N$ we have $\norm{\phi}_{a_1,a_2,k}<\infty$ where
  \begin{align}
    \norm{\phi}_{a_1,a_2,k} &= \sup_{0<t<\infty} \zeta_{a_1,a_2}(t)
    t^{k+1}
    \abs{\frac{d^k}{dt^k}\phi(t)}, \label{eq:1}\\ \zeta_{a_1,a_2}(t)
    &= \begin{cases}t^{-a_1}, &0<t\leq1,\\ t^{-a_2},
      &1<t<\infty.\end{cases} \label{eq:2}
  \end{align}
  A sequence $(\phi_j)_{j=1}^\infty \subset \M_{a_1,a_2}$ converges to
  $\phi \in \M_{a_1,a_2}$ if
  \begin{equation} \label{eq:3}
    \norm{\phi_j-\phi}_{a_1,a_2,k} \to 0
  \end{equation}
  as $j\to\infty$ for each $k=0,1,2,\ldots$.

  For $a_1 < a_2$ real or $\pm\infty$, we define $\M(a_1,a_2)$ as
  follows. A function $\phi$ is an element of $\M(a_1,a_2)$ if
  $\phi\in\M_{a,b}$ for some $a_1<a<b<a_2$. A sequence
  $(\phi_j)_{j=1}^\infty \subset \M(a_1,a_2)$ converges to it if a
  tail $(\phi_j)_{j=j_0}^\infty$, $j_0\in\N$ converges to $\phi$ in
  some fixed space $\M_{a,b}$ with $a_1<a<b<a_2$.
\end{definition}

\begin{lemma} \label{lem:t_pow_s_test_func}
  Let $ a_1, a_2$ be real numbers and $s\in \C$. Let $\phi(t)= t^{s
    -1}$ for $t>0$. Then $\phi\in \M_{a_1,a_2}$ if and only if $a_1
  \leq \Re(s) \leq a_2$. As a consequence $\phi\in\M(a_1,a_2)$ if and
  only if $a_1<\Re(s)<a_2$.
\end{lemma}
\begin{proof}
  We have
  \begin{equation}\label{eq:4}
    t^{k+1 -a_1} \left( \frac{d}{dt}\right)^k \phi(t)=
    (s-1)(s-2)\ldots(s-k)t^{s - a_1}
  \end{equation}
  and this is bounded in the interval $(0,1)$ if and only if
  $\Re(s)\geq a_1$. We see similarly that $t^{k+1-a_2}(d/dt)^k\phi(t)$
  is bounded on $(1,\infty)$ if and only if $\Re(s)\leq a_2$, which
  proves the claim.
\end{proof}

The above and the following lemma show that the $\M(a_1,a_2)$, $a_1 <
a_2$ are non-trivial. As a consequence of the following we see that
the linear functionals that we are building are in fact distributions
$\mathscr{D}'(\R_+)$. We skip the proof. It is worth noting that they
allow exponential growth, so cannot be interpreted as tempered
distributions.
\begin{lemma}\label{lem:smooth-function}
  Lets $\mathscr{D}(\R_+)$ be the space of compactly supported smooth
  test functions on $\R_+$ with the usual topology. Then
  $\mathscr{D}(\R_+) \subset \M(a_1,a_2)$ continuously for any $a_1<
  a_2$ real or infinite. The inclusion is dense.
\end{lemma}

We will introduce the space of distributions which will form a natural
domain for the Mellin transform. For intuition, see Section~4.3 in
\cite{zemanian69_gener_integ_trans}.
\begin{definition}\label{def:mellin-transformables}
  Let $a_1< a_2$ be real or infinite. By $\M'(a_1,a_2)$ we mean the
  space of continuous linear functionals on $\M(a_1,a_2)$. In detail
  $u\in \M'(a_1,a_2)$ if the following hold:
  \begin{enumerate}
  \item $\langle u, \phi \rangle$ is a complex number for each
    $\phi\in \M(a_1,a_2)$.
  \item $\langle u, c_1\phi_1 + c_2\phi_2 \rangle = c_1 \langle u,
    \phi_1 \rangle + c_2 \langle u, \phi_2 \rangle$ for all $c_1, c_2
    \in\C$ and $\phi_1,\phi_2 \in \M(a_1,a_2)$.
  \item $\langle u, \phi_j \rangle \to 0$ as $j\to\infty$ if
    $\phi_j\to0$ in $\M(a_1,a_2)$
  \end{enumerate}
  Furthermore we say that a sequence $u_j\to0$ in $\M'(a_1,a_2$) if
  $\langle u_j, \phi \rangle \to0$ in $\C$ for all
  $\phi\in\M(a_1,a_2)$.
\end{definition}

\begin{example}\label{exa:drop-from-1-func}
  Let
  \begin{equation}\label{eq:5}
    g(t) = \begin{cases}1, &0<t<1,\\ 0, &t\geq1. \end{cases}
  \end{equation}
  Then $g\in\M'(a_1,a_2)$ if and only if $a_1\geq0$ and $a_2>a_1$,
  where the latter is because we haven't allowed $a_2=a_1$ in the
  definitions. Let $a_2>a_1\geq0$, $\phi\in\M(a_1,a_2)$ and
  $(\phi_j)_{j=1}^\infty\subset\M(a_1,a_2)$ converging to $0$ in that
  space. Definition~\ref{def:test-functions} implies that there is
  $a,b$ such that $a_1<a<b<a_2$ with $\phi, \phi_j \in \M_{a,b}$ and
  the latter converging to $0$ in that same space. We have not defined
  it explicitly, but the interpretation of an ordinary function as a
  potential element of Mellin transformable distributions is by
  integrating the function multiplied by a test function. We see that
  \begin{equation}\label{eq:6}
    \langle g,\phi \rangle = \int_0^1\phi(t) \,dt = \int_0^1 t^{a-1}
    t^{0+1-a}\phi(t) \,dt \leq \int_0^1t^{a-1}\,dt \norm{\phi}_{a,b,0}
    = \frac{1}{a} \norm{\phi}_{a,b,0}.
  \end{equation}
  The same implies that $\langle g,\phi_j \rangle \leq
  a^{-1}\norm{\phi_j}_{a,b,0}\to0$ as $j\to\infty$. Hence
  $g\in\M'(a_1,a_2)$ when $a_2>a_1\geq0$.

  Next, assume that $a_2>a_1<0$ and that $g\in\M'(a_1,a_2)$. Then
  there is $p<0$ such that $a_1<p<a_2$. Let $\phi(t)=t^{p-1}$. By
  Lemma~\ref{lem:t_pow_s_test_func} we see that $\phi\in\M(a_1,a_2)$
  but by (\ref{eq:6}) it is clear that $\langle g, \phi \rangle =
  \infty$. Hence $g\notin\M(a_1,a_2)$ when $a_1<0$.
\end{example}

\smallskip
\begin{remark}
  Lemma~\ref{lem:smooth-function} implies that $\M'(a_1,a_2)\subset
  \mathscr{D}'(\R_+) $ for any $a_1<a_2$ and the inclusion is
  continuous. However the converse does not hold, because for example
  $t\rightarrow t^z$ is in $\mathscr{D}'(\R_+)\backslash\M'(a_1,a_2)$
  for any $z\in\C$ and $a_1<a_2$. Also, it looks like arbitrary
  elements of
  \begin{equation}\label{eq:7}
    L^{2,c} (\R_+)= \left\{ f:\R_+\rightarrow \C \text{ measurable}
    \,\,\middle|\,\, \int_0^\infty\abs{f(t)}^2 t^{2c -1}dt
    <\infty\right\}
  \end{equation}
  do not belong to $\M'(a_1,a_2)$. However it may happen that
  $f\in\M'(a_1,a_2)$ might satisfy $f\in L^{2,c}(\R_+)$ and then a
  Plancherel-type theorem involving Mellin transform holds.
\end{remark} 

Our strategy for this section is the following. We will define the
Mellin transform for elements of $\M'(a_1,a_2)$ and then study how the
Hilbert transform on $\R_+$ acts on them. After this we will prove
estimates for elements in $L^{2,c}(\R_+)\cap \M'(a_1,a_2)$ (which are
dense in $L^{2,c}(\R_+) $). Continuity will then imply the estimates
for $L^{2,c}(\R_+)$. Note that $t^{z-1} \in \M(a_1,a_2)$ even though
it is not in $\M'(a_1,a_2)$.

\section{The Mellin transform for distributions}\label{sec:mell-transf-distr}
We are now ready to define the Mellin transform of
$u\in\M'(a_1,a_2)$. Recall that if $u\in\M'(a_1,a_2)$ can be
represented in the form

\begin{equation}\label{eq:8}
  \langle u, \phi\rangle = \int_0^\infty f_u(t)\phi(t) \,dt, \qquad
  \phi \in \M(a_1,a_2)
\end{equation}
for some measurable function $f_u: \R_+\to\C$ then we identify $u$ and
$f_u$. Recalling that the Mellin transform of a measurable function
$f:\R_+\to\C$ is given by
\begin{equation}\label{eq:9}
  \M f(s) = \widetilde f(s) = \int_0^\infty f(t) t^{s-1} \,dt
\end{equation}
for those $s\in \C$ for which the integral converges in the sense of
Lebesgue. Inspired by these two observations we define

\begin{definition}\label{def:mellin-of-distr}
  Let $a_1,a_2 \in \{ -\infty, +\infty\} \cup \R$ with $a_1<a_2$ and
  let $u \in \M'(a_1,a_2)$. Then the Mellin transform of $u$ is
  \begin{equation}\label{eq:10}
    \M u(s)= \widetilde u(s) = \langle u, t^{s-1} \rangle
  \end{equation}
  for $s\in \C$, $a_1 < \Re(s) < a_2$.
\end{definition}

\begin{remark}\label{lem:mell-trans-welldef}
  The formula (\ref{eq:10}) is well-defined because the test function
  $\phi(t) = t^{s-1}$ is in $\M(a_1,a_2)$ whenever $a_1 < \Re(s) <
  a_2$ by Lemma~\ref{lem:t_pow_s_test_func}.
\end{remark}

It turns out that the Mellin transform of a distribution in
$\M'(a_1,a_2)$ has many nice properties. We summarize some of
them. For proofs and details, see \cite{zemanian69_gener_integ_trans}.
\begin{lemma}\label{lem:mell-transf-holom}
  If $f\in \M'(a_1,a_2)$ with $a_1 < a_2$ real numbers or $-\infty,
  +\infty$, then $s\mapsto \M f(s)$ is holomorphic in
  $a_1<\Re(s)<a_2$.
\end{lemma}

\begin{definition}\label{def:strip-of-holom}
  When we say $\M f$ has strip of holomorphicity $S$ (or $S_f$) we
  mean that
  \begin{equation}\label{eq:11}
    S = \lbrace s\in\C \mid a_1 < \Re(s) < a_2 \rbrace
  \end{equation}
  for some $a_1<a_2$ and $\M f$ is holomorphic on $S$. If $f \in
  \M'(a_1,a_2)$ with $S$ as above, we write $f\in\M'_S $ or $f\in
  \M'_{S_f}$. Also, given $a_1,a_2\in\R \cup \{-\infty, +\infty\}$, we
  denote
  \begin{equation}\label{eq:12}
    S(a_1,a_2) = \lbrace s\in\C \mid a_1<\Re(s)<a_2 \rbrace.
  \end{equation}
\end {definition}

\bigskip
The Mellin transform for distributions has several properties.
\begin{theorem}\label{thm:props-of-mellin-trans}
  In the following we assume that $f\in\M'_{S_f}$ and
  $g\in\M'_{S_g}$. It holds that:
  \begin{enumerate}
  \item\label{thm:props-of-mellin-trans:item0} If $n\in\N$ then
    $(-t\,d/dt)^n f \in\M'_{S_f}$ and $\M[(-t\,d/dt)^nf](s) = s^n
    \M[f](s)$.
  \item\label{thm:props-of-mellin-trans:item1} If $S_f \cap S_g \neq
    \emptyset$ and $\M f = \M g$ on $S_f\cap S_g$ then $f=g$ as
    distributions in $\M'_{S_f\cap S_g}$ and a fortiori in
    $\mathscr{D}'(\R_+)$.
  \item\label{thm:props-of-mellin-trans:item2} A function $F:S_f\to\C$
    is the Mellin transform of some $f\in\M'_{S_f}$ if and only if
    \begin{enumerate}[a)]
    \item $F$ is holomorphic in $S_f$, and
    \item for any closed substrip of $S_f$ of the form $\alpha_1 \leq
      \Re(s) \leq \alpha_2$ there is a polynomial $P$ such that
      $\abs{F(s)} \leq P (|s|) $ on that strip.
    \end{enumerate}
  \item\label{thm:props-of-mellin-trans:item3} Let $S_f \cap S_g =
    \{s\in \C \mid a_1 < \Re(s) < a_2\}$. Then
    \begin{equation}\label{eq:13}
      \M[f\vee g](s) = \M f(s)\, \M g(s), \qquad a_1 <\Re(s)< a_2
    \end{equation}
    where
    \begin{equation}\label{eq:14}
      (f\vee g)(\tau) = \int_0^\infty f(t)
      g\left(\frac{\tau}{t}\right) \frac{dt}{t}, \qquad\tau> 0
    \end{equation}
    if $f$ and $g$ are integrable functions and otherwise
    \begin{equation}\label{eq:15}
      \langle f\vee g, \theta \rangle = \langle f, \psi\rangle, \qquad
      \psi (t) = \langle g, \theta_t \rangle
    \end{equation}
    for $\theta \in \M(a_1,a_2)$, $t>0$ and $\theta_t(\tau) =
    \theta(t\tau)$.
  \end{enumerate}
\end{theorem}
Recall from Section~\ref{sec:hilb-mell-transf} that $f\vee g$ in
(\ref{eq:14}) is the convolution of the multiplicative group
$(\R_+,\cdot)$ and $dt/t$ is its Haar measure.

\medskip
The following gives an inversion formula for the Mellin transform.
\begin{theorem}\label{thm:inverse-mellin}
  If $F : S(a_1, a_2)\to\C$ is holomorphic and satisfies $\abs{F(s)}
  \leq K\abs{s}^{-2}$ for some finite constant $K$, and we set
  \begin{equation}\label{eq:16}
    f(t) = \frac{1}{2\pi
      i}\int_{\sigma-i\infty}^{\sigma+i\infty}F(s)t^{-s}\,ds,
  \end{equation}
  for a fixed $\sigma\in(a_1,a_2)$, then $f:\R_+\to\C$ is continuous,
  does not depend on the choice of $\sigma$ and is in
  $\M'(a_1,a_2)$. Furthermore $\M f=F$ on $S(a_1,a_2)$.
\end{theorem}

The following corollary is Theorem~4.4.1 in
\cite{zemanian69_gener_integ_trans}. In that reference, it is used to
prove the result that corresponds to
Item~\ref{thm:props-of-mellin-trans:item2} of
Theorem~\ref{thm:props-of-mellin-trans} of our
article\footnote{Strictly speaking, this applies to the corresponding
results for the Laplace transform. The results from the Mellin
transform are only stated.}, and it gives another inversion formula
for the cases where the theorem above cannot be applied. Namely, if
$F$ has a singularity on the border of $S(a_1,a_2)$.
\begin{corollary}\label{cor:invert-poly-growth}
  Let $F : S(a_1, a_2)\to\C$ be holomorphic and $Q:\C\to\C$ be a
  polynomial that has no zeroes in $S(a_1,a_2)$ such that
  \begin{equation}\label{eq:17}
    \abs{\frac{F(s)}{Q(s)}} \leq \frac{K}{\abs{s}^2}, \qquad b_1 <
    \Re(s) < b_2
  \end{equation}
  for some $a_1 < b_1 < b_2 < a_2$ and a finite constant $K$. Set
  \begin{equation}\label{eq:18}
    g(t) = \frac{1}{2\pi i}\int_{\sigma-i\infty}^{\sigma+i\infty}
    \frac{F(s)}{Q(s)} t^{-s}\, ds,
  \end{equation}
  for some $b_1<\sigma<b_2$. Then $g:\R_+\to\C$ is continuous, belongs
  to $\M'(b_1,b_2)$ as does $f(t)=Q(-t\,d/dt)g(t)$ too. Furthermore
  $\M f = F$ on $S(b_1,b_2)$.
\end{corollary}

\section{The Hilbert transform}\label{sec:hilbert-transform}
We will need to know the Mellin transform of the distribution
\begin{equation}\label{eq:19}
  \left\langle H, \phi \right\rangle = \frac{1}{\pi} \lim
  _{\varepsilon\to 0+}\left( \int_0^{1 - \varepsilon} + \int
  _{1+\varepsilon} ^\infty\right) \frac{\phi(t)}{1-t}dt,
\end{equation}
namely $H=\pi^{-1}/(1-t)$ in the principal value sense. It is almost
the kernel of the Hilbert transform of a function vanishing on $\R_-$
\begin{equation}\label{eq:20}
  \Ht f(x) = \frac{1}{\pi}\pv \int_0^\infty \frac{f(y)}{x-y}\,dt.
\end{equation}
In fact, formally
\begin{equation}\label{eq:21}
  \Ht f(x) = -\frac{1}{\pi}\pv\int_0^\infty \frac{1}{1-t}
  f\left(\frac{x}{t}\right) \,\frac{dt}{t} = -(H \vee f)(x),
\end{equation}
which can be deduced from (\ref{eq:20}) by change integration
variables $y=x/t$, $dy = - x \,dt / t^2$.

\begin{lemma}\label{lem:pv-Mellin-transformable}
  The distribution $1/(1-t)$ in the principal value sense belongs to
  $\M'(0,1)$. Furthermore it can be written as
  \begin{equation}\label{eq:22}
    \left\langle \frac1{1-t},\phi \right\rangle = \left(\int_0^{1/2} +
    \int_{3/2}^\infty\right) \frac{\phi(t)}{1-t} dt - \int_{1/2}^{3/2}
    \frac{\phi(t)-\phi(1)}{t-1} dt
  \end{equation}
  where $1/(1-t)$ is interpreted as a pointwise function on the
  right-hand side. Lastly, there is a finite $C$ such that
  $\abs{\langle 1/(1-t),\phi \rangle} \leq C (\norm{\phi}_{0,1,0} +
  \norm{\phi}_{0,1,1})$.
\end{lemma}
\begin{proof}
  Let us denote $u = 1/(1-t)$ and recall that the distribution
  pairings are done with the principal value. We will first prove that
  $\langle u, \phi \rangle\in\C$ for $\phi\in\M(0,1)$. The latter
  means there are $0<a<b<1$ such that $\phi\in\M_{a,b}$. In particular
  (\ref{eq:1}) implies that $\norm\phi_{a,b,0}$ and
  $\norm\phi_{a,b,1}$ are finite. Let $h(t)=1$ for $1/2<t<3/2$ and
  $h(t)=0$ otherwise. Then
  \begin{equation}\label{eq:23}
    \langle u, \phi \rangle = \lim_{\epsilon \rightarrow 0} \left(
    \int _0^{1 -\epsilon} + \int_{1 +\epsilon}^\infty\right) \left(
    \frac{\phi(t) - \phi (1)h(t)}{ 1 -t} + \frac{\phi(1) h(t)}{1
      -t}\right) dt
  \end{equation}
  with $s=2-t$ we see that
  \begin{align*}
    \int_{1/2}^{1-\epsilon} \frac{\phi(1) h(t)}{1-t} dt &=
    \phi(1)\int_{1/2}^{1-\epsilon} \frac{dt}{1-t} = \phi(1)
    \int_{3/2}^{1+\epsilon} \frac{-ds}{-1+s} =
    \phi(1)\int_{3/2}^{1+\epsilon} \frac{ds}{1-s}\\ &=
    -\int_{1+\epsilon}^{3/2} \frac{\phi (1) h(s)}{1-s}ds
  \end{align*}
  and so the last integral in (\ref{eq:23}) vanishes. For the first
  integral recall that $\phi$ is smooth. Hence the secant $(\phi(t) -
  \phi(1)) / (t-1)$ is a continuous function of $t$. We see that
  \begin{align}\label{eq:24}
    \left( \int_0^{1-\epsilon} + \int_{1+\epsilon}^\infty \right)
    \frac{\phi(t) - \phi(1) h(t)}{1-t} dt & = \left( \int_0^{1/2} +
    \int_{3/2}^\infty \right) \frac{\phi(t)}{1-t} dt \\ \nonumber &+
    \left( \int_{1/2}^{1 -\epsilon} + \int_{1+\epsilon}^{3/2}
    \right)\frac{\phi(t) - \phi(1)}{1-t} dt .
  \end{align}
  This proves (\ref{eq:22}), as $\epsilon$ can be let equal to zero as
  the secant is continuous. The first integrand is continuous on
  $(0,1/2)\cup(3/2,\infty)$. It is also integrable since
  \begin{align}
    \int_0^{1/2} \abs{\frac{\phi(t)}{1-t}}dt &\leq \int_0^{1/2}
    t^{a-1} t^{1-a} \abs{\phi(t)} \cdot 2 dt \leq 2\intlim{0}{1/2}
    \frac{t^a}{a} \sup_{0<t<1/2} t^{1-a}\abs{\phi(t)} \notag\\ &\leq
    \frac{2^{1-a}}{a} \norm{\phi}_{a,b,0} < \infty. \label{eq:25}
  \end{align}
  Similarly
  \begin{align}
    \int_{3/2}^\infty \abs{\frac{\phi(t)}{1-t}}dt &\leq
    \int_{3/2}^\infty \frac{t^{b-1}}{t-1}t^{1-b} \abs{\phi(t)} dt \leq
    \norm{\phi}_{a,b,0} \int_{3/2}^\infty \frac{t^{b-1}}{t-1} dt
    \notag\\ &\leq \norm{\phi}_{a,b,0} 3\int_{3/2}^\infty
    \frac{t^{b-1}}{t} dt = \norm{\phi}_{a,b,0} 3 \intlim{3/2}{\infty}
    \frac{t^{b-1}}{b-1} \notag\\ &=\frac{3(3/2)^{b-1}}{1-b}
    \norm{\phi}_{a,b,0} <\infty.\label{eq:26}
  \end{align}
  because $1/(t-1) \leq 3/t$ for $t\geq3/2$.

  For the second integrand in (\ref{eq:24}) note that
  \begin{equation}\label{eq:27}
    \abs{\frac{\phi(t)-\phi(1)}{t-1}} = \abs{\phi'(\xi)} \leq
    \sup_{1/2<t<3/2}\abs{\phi'(t)}\leq C\norm{\phi}_{a,b,1} < \infty
  \end{equation}
  for some finite constant $C$. Hence we can take the limit and have
  \begin{equation}\label{eq:28}
    \lim_{\epsilon\to0} \left( \int_{1/2}^{1-\epsilon} +
    \int_{1+\epsilon}^{3/2} \right) \frac{\phi(t)-\phi(1)}{1-t} dt =
    -\int_{1/2}^{3/2} \frac{\phi(t)-\phi(1)}{t-1} dt
  \end{equation}
  which is bounded by
  \begin{equation}\label{eq:29}
    \int_{1/2}^{3/2} C\norm{\phi}_{a,b,1}dt = C \norm{\phi}_{a,b,1} <
    \infty.
  \end{equation}
  Hence $\langle u,\phi \rangle \in \C $ for any $\phi\in \M_{a,b}$
  with $0<a<b<1$. Similarly, by our calculation so far we have
  $\abs{\langle u,\phi \rangle} \leq C \big( \norm{\phi}_{a,b,0} +
  \norm{\phi}_{a,b,1}\big)$ for a finite constant $C$ whenever $\phi
  \in \M_{a,b}$. By (\ref{eq:1}) we can decrease $a$ and increase $b$
  to get
  \[
  \abs{\langle u,\phi \rangle} \leq C (\norm{\phi}_{0,1,0} +
  \norm{\phi}_{0,1,1})
  \]
  for any $\phi\in\M_{a,b}$. Because this holds for arbitrary
  $0<a<b<1$, by Definition~\ref{def:test-functions} the same estimate
  holds for all $\phi\in\M(0,1)$. So the estimate in our claim is
  proven.

  Now, let $(\phi_j)_{j=1}^\infty \to 0$ in $\M(0,1)$. This means that
  there is $0<a<b<1$ such that $(\phi_j)_{j=1}^\infty \subset
  \M_{a,b}$ and $\norm{\phi_j}_{a,b,k}\to0$ as $j\to\infty$ for each
  $k\in \N$. Thus
  \[
  \abs{\langle u,\phi_j \rangle} \leq C \left( \norm{\phi_j}_{a,b,0} +
  \norm{\phi_j}_{a,b,1}\right) \to 0
  \]
  and continuity is proven. The linearity property is trivial. Hence
  $u\in\M'(0,1)$.
\end{proof}

\begin{lemma}\label{lem:pv-Mellin-trans}
  We have $\M[1/(1-t)](s) = \pi\cot(\pi s)$ in the principal value
  sense for $0<\Re(s)<1$.
\end{lemma}
\begin{proof}
  The distribution is in $\M(0,1)$. All we need to do is to calulate
  \begin{equation}\label{eq:30}
    \pv\int_0^\infty \frac{t^{s-1}}{1-t} dt.
  \end{equation}
  Refer to Example~8.24.II in \cite{pap99_compl_analy_examp_exerc},
  especially pages 219--220 for the calculations.
\end{proof}

\begin{definition}\label{def:hilbert-transform}
  For $f\in\M'(a,b)$ with $0\leq a < b \leq 1$ define the Hilbert
  transform by
  \begin{equation}\label{eq:31}
    \Ht f = - H\vee f
  \end{equation}
  where $H$ is defined in (\ref{eq:19}) and $\vee$ in (\ref{eq:15}).
\end{definition}

\begin{lemma}\label{lem:hilbert-well-def}
  The Hilbert transform is a well-defined element of $\M'(a,b)$ and if
  $f$ is smooth and compactly supported in $\R_+$ we have
  (\ref{eq:20}).
\end{lemma}
\begin{proof}
  Lemma~\ref{lem:pv-Mellin-transformable} implies that $H\in\M'(0,1)$,
  and so Theorem~4.6.1 and the paragraph after it in
  \cite{zemanian69_gener_integ_trans} imply that $H\vee f \in
  \M'(a,b)$ when $f\in\M'(a,b)$.

  Let $f$ be smooth and compactly supported. We will use Theorem~4.6.2
  by Zemanian \cite{zemanian69_gener_integ_trans}. In the sense of
  distributions on $\R_+$, we have $H\vee f$ equal to the following
  smooth function
  \begin{equation}\label{eq:Hilbert-smooth}
    g(x) := \left\langle H, \frac1t f\left(\frac{x}{t}\right)
    \right\rangle_t = \lim_{\varepsilon\to0} \frac1\pi
    \left(\int_0^{1-\varepsilon} + \int_{1+\varepsilon}^\infty\right)
    \frac{f(x/t)}{1-t} \frac{dt}{t}.
  \end{equation}
  A change of integration variables $t=x/y$, $dt = -xdy/y^2$ gives
  \begin{equation}\label{eq:Hilbert-weird-pv}
    -g(x) = \lim_{\varepsilon\to0}\frac1\pi
    \left(\int_0^{x/(1+\varepsilon)} +
    \int_{x/(1-\varepsilon)}^\infty\right) \frac{f(y)}{x-y} dy
  \end{equation}
  which equals (\ref{eq:20}) by the following.

  It remains to show that
  \begin{equation}\label{eq:pv_change_variables_ok}
    \lim_{\varepsilon\to 0} \left(\int_0^{x/(1+\varepsilon)} +
    \int_{x/(1-\varepsilon)}^\infty\right) \frac{f(x/t)}{t -1}
    \frac{dt}{t} = \lim_{\varepsilon\to 0} \left(\int_0^{x
      -\varepsilon} + \int_{x +\varepsilon}^\infty\right) \frac{1}{x -
      y}f(y)dy
  \end{equation}
  for all $x$. We obtain
  \begin{align}
    \left(\int_0^{x/(1+\varepsilon)} +\int_{x/(1-\varepsilon)}^\infty
    \right) \frac{f(x/t)}{t-1} \frac{dt}{t} &= \left(
    \int_0^{x-\varepsilon x} +\int_{x+\varepsilon x}^\infty\right)
    \frac{1}{x-y} f(y) dy \notag \\ + &\left( \int_{x-\varepsilon
      x}^{x/(1+\varepsilon)} + \int_{x/(1-\varepsilon)}^{x+\varepsilon
      x} \right) \frac{1}{x-y} f(y) dy.
  \end{align}
  For any fixed $x\in (0 , \infty)$ the first terms above clearly
  converges to the right-hand side of
  (\ref{eq:pv_change_variables_ok}). For $0<\varepsilon < 1/2$ we have
  \begin{align}
    0 < \frac{1}{1 +\varepsilon} - (1 - \varepsilon) &\leq
    \varepsilon^2\label{eq:epsilon-est1}\\ 0 < \frac{1}{1 -
      \varepsilon} - (1 + \varepsilon) &\leq 2
    \varepsilon^2\label{eq:epsilon-est2}\\ 0 < 1 -
    \frac{1}{1+\varepsilon} &\leq
    \frac\varepsilon2.\label{eq:epsilon-est3}
  \end{align}
  Hence in the term
  \[
  \int_ {x -\varepsilon x}^{x/(1 +\varepsilon)} \frac{1}{x - y}f(y)dy
  \]
  we have $\abs{ x -y} \geq \varepsilon x /2$ by
  (\ref{eq:epsilon-est3}).  The length of the integration interval is
  less than $\varepsilon^2 x$ by (\ref{eq:epsilon-est1}). It follows
  that the absolute value of this term has the upper bound
  \[
  \varepsilon^2 x \cdot \frac{2}{\varepsilon x} \max\abs{f} \leq
  2\varepsilon \max\abs{f}
  \]
  and this tends to $0$ as $\varepsilon\to0$. Similarly, using
  (\ref{eq:epsilon-est2}) and (\ref{eq:epsilon-est3}), one can show
  that
  \[
  \abs{\int_{x/(1 -\varepsilon)} ^{1+\varepsilon x} \frac{1}{x -y}f(y)
    dy}\leq 2 \varepsilon \max\abs{f}
  \]
  We have thus shown
  \[
  \lim_{\varepsilon\to 0} \left(\int_0^{1 -\varepsilon} + \int_{1
    +\varepsilon}^\infty\right) \frac{f(x/t)}{t -1} \frac{dt}{t} =\Ht
  f(x)
  \]
  for every $x\in (0, \infty)$.
\end{proof}

The results of this section can be summarized as follows.
\begin{theorem}\label{thm:Hilbert-transf-summary}
  The Hilbert transform $\Ht$ applied to test functions $f\in\mathscr
  D(\R_+)$ can be written as
  \begin{equation}\label{eq:Hilbert-forms}
    \Ht f(x) = \pv \frac1\pi \int_0^\infty \frac{f(y)}{x-y} dy = - \pv
    \frac1\pi \int_0^\infty \frac{f(x/t)}{1-t} \frac{dt}{t}.
  \end{equation}
  Applied to a distribution $u\in\M'(a,b)$ with $0\leq a<b\leq1$, it
  is an element of $\M'(a,b)$ defined by $\Ht u = -H\vee u$ with
  \begin{align}
    \langle H,\phi \rangle &= \pv \frac1\pi \int_0^\infty
    \frac{\phi(t)}{1-t} dt \label{eq:Hvee-def1} \\ \langle H\vee u ,
    \theta\rangle &= \langle H, \psi \rangle, \qquad \psi(t) = \langle
    u,\theta_t \rangle,\, \theta_t(s) =
    \theta(ts) \label{eq:Hvee-def2}
  \end{align}
  for $\theta\in\M(a,b)$. Lastly, if $u\in\M'(a,b)$ with $0\leq a<b
  \leq1$ then
  \begin{equation}\label{eq:Hilbert-Mellin}
    \M[\Ht u](s) = -\cot(\pi s)\M[u](s)
  \end{equation}
  for $a < \Re(s) < b$.
\end{theorem}
\begin{proof}
  The equations (\ref{eq:Hilbert-forms}), (\ref{eq:Hvee-def1}) and
  (\ref{eq:Hvee-def2}) are a restatement of
  Definition~\ref{def:hilbert-transform} and
  Lemma~\ref{lem:hilbert-well-def}, the latter of which gives the
  mapping properties for $\Ht$ mentioned in the
  claim. Equation~(\ref{eq:Hilbert-Mellin}) follows from (\ref{eq:13})
  in Theorem~\ref{thm:props-of-mellin-trans} and
  Lemma~\ref{lem:pv-Mellin-trans}.
\end{proof}

\section{Inhomogeneous Hilbert transform on a half-line}\label{sec:inhom-hilb-transf}
In this section we will prove that the solution $\rho$ to the equation
\begin{equation}\label{eq:Hrho-e}
  \Ht\rho = e, \qquad \R_+
\end{equation}
has a blow-up singularity at $x=0$ when $e$ is general but in a
suitable function space.

\begin{lemma}\label{lem:Hrho-zero-at-half}
  Let $0\leq a \leq \alpha < \beta \leq b \leq 1$ and let
  $e\in\M'(a,b)$, $\rho\in\M'(\alpha,\beta)$. Assume that
  (\ref{eq:Hrho-e}) holds in $\M'(\alpha,\beta)$. If $1/2 \in
  (\alpha,\beta)$ then $\M[e](1/2) = 0$.
\end{lemma}
\begin{proof}
  Take the Mellin transform of (\ref{eq:Hrho-e}). By
  Theorem~\ref{thm:Hilbert-transf-summary} we have
  \[
  -\cot(\pi s)\M[\rho](s) = \M[e](s)
  \]
  for $\alpha < \Re(s) < \beta$. In particular this hold at $s=1/2$ if
  this point belongs to the interval $(\alpha,\beta)$. Since $\M'(a,b)
  \subset \M'(\alpha,\beta)$, we have
  $\rho,e\in\M'(\alpha,\beta)$. Then by
  Lemma~\ref{lem:mell-transf-holom} both $\M[\rho]$ and $\M[e]$ are
  holomorphic in a complex neighbourhood of $s=1/2$; in particular
  $\M[\rho](1/2)$ is a well-defined finite complex number. Since
  $\cot(\pi/2)=0$, a value not changed by multiplying with a complex
  number, we have $\M[e](1/2)=0$.
\end{proof}

\begin{lemma}\label{lem:tangent-estimates}
  Let $x,y\in\R$. If $x$ is at least $\varepsilon>0$ distance from
  $1/2+\Z$ then
  \begin{equation}\label{eq:tan-est-x-fixed}
    \abs{\tan\big(\pi(x+iy)\big)}^2 \leq
    \big(\cos\pi(1-2\varepsilon)+1\big)^{-2}
  \end{equation}
  which is finite when such an $x$ exists. Otherwise, if $\abs{y}=M>0$
  we have
  \begin{equation}\label{eq:tan-est-y-fixed}
    \abs{\tan\big(\pi(x+iy)\big)}^2 \leq \big(1-(\cosh(2\pi
    M))^{-1}\big)^{-2}
  \end{equation}
  which is always finite, and at most $4$ when $M>1/\pi$.
\end{lemma}
\begin{proof}
  We start with the trigonometric identity
  \begin{equation}\label{eq:tan-id}
    \tan\big(\pi(x+iy)\big) = \frac{\sin(2\pi x) + i \sinh(2\pi
      y)}{\cos(2\pi x) + \cosh(2\pi y)}.
  \end{equation}
  Taking the square of the modulus and using $\sinh^2(2\pi y) =
  \cosh^2(2\pi y)-1$ we get
  \begin{equation}\label{eq:tan-squared-id}
    \abs{\tan\big(\pi(x+iy)\big)}^2 = \frac{\sin^2(2\pi x) +
      \cosh^2(2\pi y)-1}{\big(\cos(2\pi x) + \cosh(2\pi y)\big)^2}.
  \end{equation}

  If $x$ is at least distance $\varepsilon>0$ from $1/2+\Z$, we must
  have $0<\varepsilon\leq1/2$. Then $\cos(2\pi x) \geq
  \cos(2\pi(1/2-\varepsilon))$, and since $\cosh(2\pi y)\geq1$ and
  $\sin^2(2\pi x)\leq1$, we get
  \[
  \abs{\tan\big(\pi(x+iy)\big)}^2 \leq \frac{\cosh^2(2\pi
    y)}{\big(\cos\pi(1-2\varepsilon) + \cosh(2\pi y)\big)^2}.
  \]
  This implies (\ref{eq:tan-est-x-fixed}) after reducing the fraction
  by its numerator, noting that $-1<\cos\pi(1-2\varepsilon)\leq0$ and
  using $\cosh(2\pi y) \geq 1$.

  Now, if we just have $\abs{y}=M>0$, we can estimate $\cos(2\pi
  x)\geq-1$ and $\sin^2(2\pi x) \leq 1$ in (\ref{eq:tan-squared-id})
  to get
  \[
  \abs{\tan\big(\pi(x+iy)\big)}^2 \leq \frac{\cosh^2(2\pi y)}{\big(-1
    + \cosh(2\pi y)\big)^2}.
  \]
  However since $M>0$ and the evenness of the hyperbolic cosine, we
  have $\cosh(2\pi y)=\cosh(2\pi M)>1$ so the right-hand side is a
  finite constant depending on $M$. The last claim follows since
  $M>1/\pi$ implies that $\cosh(2\pi M) > 2$.
\end{proof}

\begin{lemma}\label{lem:explicit-solution}
  Let $e\in\M'(a,b)$ for some $0\leq a<b\leq1$.  If
  \[
  a<b \leq 1/2, \qquad\text{or}\qquad 1/2\leq a<b,
  \qquad\text{or}\qquad \M[e](1/2) = 0
  \]
  then there is $\rho\in\M'(a,b)$ satisfying $\Ht\rho=e$. Furthermore,
  for any $\alpha,\beta,c$ with $a<\alpha<c<\beta<b$ for this $\rho$
  it holds that
  \begin{equation}
    \rho(t) = \frac{-1}{2\pi i} \left(-t\, d/dt\right)^{m+2}
    \int_{c-i\infty}^{c+i\infty} s^{-m-2} \tan(\pi s) \M[e](s) t^{-s}
    ds
  \end{equation}
  in $\M'(\alpha,\beta)$. Here $m\in\mathbb N$ can be any number for
  which there is a polynomial $P$ of degree $m$ such that
  $\abs{\M[e](s)}\leq P(\abs{x})$ on $S(\alpha,\beta)$.

  In the case where
  \[
  a < 1/2 < b, \qquad\text{and}\qquad \M[e](1/2)\neq0
  \]
  there are no solutions in any $\M'(\alpha,\beta)$ with
  $\alpha<1/2<\beta$. Instead there is $\rho_-\in\M'(a,1/2)$ and
  $\rho_+\in\M'(1/2,b)$ such that $\Ht\rho_\pm=e$ in $\M'(a,1/2)$ and
  $\M'(1/2,b)$, respectively. They satisfy
  \begin{align}
    \rho_-(t) &= \frac{-1}{2\pi i} \left(-t\, d/dt\right)^{m+2}
    \int_{c_--i\infty}^{c_-+i\infty} s^{-m-2} \tan(\pi s) \M[e](s)
    t^{-s} ds,\label{eq:rho-m-formula}\\ \rho_+(t) &= \frac{-1}{2\pi
      i} \left(-t\, d/dt\right)^{m+2} \int_{c_+-i\infty}^{c_++i\infty}
    s^{-m-2} \tan(\pi s) \M[e](s) t^{-s} ds\label{eq:rho-p-formula}
  \end{align}
  in $\M'(\alpha_-,\beta_-)$ and $\M'(\alpha_+,\beta_+)$,
  respectively, for any $a<\alpha_- < c_- < \beta_- < 1/2$ and $1/2 <
  \alpha_+ < c_+ < \beta_+ < b$. Here $m\in\mathbb N$ can be any
  number for which there is a polynomial $P$ of degree $m$ such that
  $\abs{\M[e](s)}\leq P(\abs{x})$ on $S(\alpha_-,\beta_+)$.
\end{lemma}
\begin{proof}
  Write $F(s) = -\tan(\pi s)\M[e](s)$. Then $F:S(a,b)\to\C$ is
  holomorphic everywhere except at $s=1/2$ if $\M[e](1/2)\neq0$. We
  want to use the Mellin transform inversion formula. For that we need
  to show an estimate for $\abs{F(s)}$ that holds uniformly in a
  vertical strip of the complex plane.

  Let us first consider the case ``$a<b \leq 1/2$, $1/2\leq a<b$, or
  $\M[e](1/2) = 0$''. In that case $F$ is holomorphic on $S(a,b)$. We
  want to let $\rho$ be the inverse Mellin transform of $F$, but for
  that we need to prove some estimates first, so that we can use
  Item~\ref{thm:props-of-mellin-trans:item2} of
  Theorem~\ref{thm:props-of-mellin-trans}.

  Consider an arbitrary closed substrip $\alpha_1 \leq \Re(s) \leq
  \alpha_2$ of $S(a,b)$. If it contains $s=1/2$ then our assumptions
  imply that $\M[e](1/2)=0$, in which case $\abs{F(1/2)}<\infty$ so
  there is $r>0$ and $C<\infty$ such that $\abs{F(s)}<C$ when
  $\abs{s-1/2}<r$. When $\abs{s-1/2}\geq r$ we have
  \[
  \abs{\tan(\pi s)} \leq C_r
  \]
  by Lemma~\ref{lem:tangent-estimates}. Furthermore there is some
  polynomial $P$ such that
  \begin{equation}\label{eq:mellin-e-polynomial}
    \abs{\M[e](s)} \leq P(\abs{s})
  \end{equation}
  on that closed vertical strip by
  Item~\ref{thm:props-of-mellin-trans:item2} of
  Theorem~\ref{thm:props-of-mellin-trans}. In both cases whether
  $\alpha_1\leq1/2\leq\alpha_2$ or not, there is thus some finite
  constant $K$ for which
  \begin{equation}\label{eq:F-bound-case-1}
    \abs{F(s)}\leq K\big( 1 + P(\abs{s}) \big)
  \end{equation}
  when $\alpha_1 \leq \Re(s) \leq \alpha_2$. Because this is an
  arbitrary vertical closed substrip of $S(a,b)$ then by the same item
  of the same theorem we see that there is $\rho\in\M'(a,b)$ such that
  $\M\rho = F$ on $S(a,b)$.

  Next, by the Mellin transform formula for the Hilbert transform of
  Theorem~\ref{thm:Hilbert-transf-summary}, we have
  \begin{equation}\label{eq:rho-sol-case1}
    \M [\Ht\rho] (s) = -\cot(\pi s)(-\tan(\pi s)) \M[e](s) = \M[e](s)
  \end{equation}
  for $s\in S(a,b)$. So by the uniqueness of the inverse Mellin
  transform (Item~\ref{thm:props-of-mellin-trans:item1} of
  Theorem~\ref{thm:props-of-mellin-trans}) we have $\Ht\rho=e$ in
  $\M'(a,b)$.

  Next, let $\alpha,\beta,c$ be as in the assumptions. Then, as in
  (\ref{eq:mellin-e-polynomial}), we see that there is a polynomial
  $P$ such that $\abs{\M[e](s)}\leq P(\abs{s})$ for
  $\alpha\leq\Re(s)\leq\beta$. Let $Q(x)=x^{2+m}$, $m=\deg P$. By the
  estimate for $\abs{F(s)}$ from (\ref{eq:F-bound-case-1}) we have
  \[
  \abs{\frac{F(s)}{Q(s)}} \leq \frac{K \big( 1 + P(\abs{s})
    \big)}{\abs{s}^2 \abs{s}^m} \leq \frac{C}{\abs{s}^2}
  \]
  when $\alpha\leq \Re(s)\leq\beta$. If we set
  \[
  f(t) = \frac{1}{2\pi i} (-t d/dt)^{m+2} \int_{c-i\infty}^{c+i\infty}
  s^{-m-2} F(s) t^{-s} ds
  \]
  then the integral gives a continuous function $\R_+\to\C$ that's in
  $\M'(\alpha,\beta)$, and also $f\in \M'(\alpha,\beta)$ satisfies $\M
  f = F$ in $S(\alpha,\beta)$ by
  Corollary~\ref{cor:invert-poly-growth}. Because $\M\rho =F$ in
  $S(a,b)$ we have $f=\rho$ in $\M'(\alpha,\beta)$ by
  Item~\ref{thm:props-of-mellin-trans:item1} of
  Theorem~\ref{thm:props-of-mellin-trans}. This concludes the proof of
  the first case.

  \smallskip
  In the case where $a<1/2<b$ and $\M[e](1/2)\neq0$ there are no
  solutions in $\M'(a,b)$ by Lemma~\ref{lem:Hrho-zero-at-half}. Note
  also that in this case $F$ is holomorphic in $S(a,1/2)$ and
  $S(1/2,b)$ while having a singularity at $s=1/2$. Consider the
  closed vertical strips $\alpha_1\leq\Re(s)\leq\alpha_2$ and
  $\beta_1\leq\Re(s)\leq\beta_2$ for arbitrary
  $a<\alpha_1<\alpha_2<1/2$ and $1/2<\beta_1<\beta_2<b$.  As in the
  first part of the proof, we see that
  \[
  \abs{\tan(\pi s)} \leq C_{\alpha_2,\beta_1}
  \]
  by Lemma~\ref{lem:tangent-estimates} when $s$ belongs to either of
  these two closed strips because $\alpha_2<1/2$ and $1/2<\beta_1$. As
  before, we have
  \[
  \abs{\M[e](s)} \leq P(\abs{s})
  \]
  on $\alpha_1\leq\Re(s)\leq\beta_2$ by
  Item~\ref{thm:props-of-mellin-trans:item2} of
  Theorem~\ref{thm:props-of-mellin-trans}. These two estimates give a
  polynomial upper bound for $\abs{F(s)}$ on
  $\alpha_1\leq\Re(s)\leq\alpha_2$ and on
  $\beta_1\leq\Re(s)\leq\beta_2$ as in the first part of the
  proof. Since the closed substrips were arbitrary, these then imply
  the existence of $\rho_-\in\M'(a,1/2)$ and $\rho_+\in\M'(1/2,b)$
  satisfying $\Ht\rho_\pm=e$ in $\M'(a,1/2)$ and $\M'(1/2,b)$,
  respectively. With identical deductions as in the first part of the
  poof, we see the integral representation formulas for $\rho_\pm$ in
  $\M'(\alpha_\pm,\beta_\pm)$.
\end{proof}

\begin{lemma}\label{lem:uniqueness-in-part}
  Let $0\leq a<b\leq 1$ and $\rho_1,\rho_2\in\M'(a,b)$. If
  \[
  \Ht\rho_1 = \Ht\rho_2
  \]
  then $\rho_1=\rho_2$ in $\M'(a,b)$.
\end{lemma}
\begin{proof}
  By taking the Mellin transform of the equation, and using the
  transformation properties of the Hilbert transform from
  Theorem~\ref{thm:Hilbert-transf-summary} we see that
  \[
  -\cot(\pi s)\M[\rho_1](s) = -\cot(\pi s)\M[\rho_2](s)
  \]
  for $s\in S(a,b)$. When $s\neq1/2$ we can divide by the cotangent
  and get
  \[
  \M[\rho_1](s) = \M[\rho_2](s)
  \]
  for $s\in S(a,b)\setminus\{1/2\}$. But $\rho_1-\rho_2\in\M'(a,b)$ so
  $\M[\rho_1] - \M[\rho_2]$ is holomorphic in $S(a,b)$. Hence the
  equality holds in the whole $S(a,b)$. According to the properties of
  Mellin transform in Theorem~\ref{thm:props-of-mellin-trans} we have
  $\rho_1=\rho_2$ in $\M'(a,b)$.
\end{proof}

\begin{lemma}\label{lem:tan-residue}
  The residue of $\tan(\pi s)$ at $s=1/2$ is given by
  \[
  \operatorname{Res}\big(\tan(\pi s),1/2\big) = -\frac1\pi.
  \]
\end{lemma}
\begin{proof}
  We have $\sin(\pi/2)=1$ and $\cos(\pi/2)=0$ so the residue is given
  by the cosine. Then
  \begin{align*}
    \lim_{s\to1/2}\frac{s-1/2}{\cos(\pi s)} &= \lim_{s\to1/2}
    \frac1\pi \frac{\pi s-\pi/2}{\cos(\pi s)-\cos(\pi/2)} = \frac1\pi
    \lim_{\xi\to\pi/2}\frac{1}{\frac{\cos\xi-\cos(\pi/2)}{\xi-\pi/2}}
    \\ &= \frac1\pi \frac{1}{\cos'(\pi/2)} = -\frac1\pi
    \frac{1}{\sin(\pi/2)} = -\frac1\pi.
  \end{align*}
  Thus $\operatorname{Res}(\tan(\pi s),1/2) =
  \lim_{s\to1/2}(s-1/2)\tan(\pi s) = -1/\pi$.
\end{proof}

\begin{lemma}\label{lem:cauchy-integral}
  Let $0<\alpha<1/2<\beta<1$ and $f:S(\alpha,\beta)\to\C$ be
  holomorphic with $\abs{f(s)}\leq C s^m$ for some $m\in\mathbb
  N$. For $\alpha<c_-<1/2<c_+<\beta$ define
  \begin{align*}
    \bar\rho_-(t) &= \frac{-1}{2\pi i}
    \int_{c_--i\infty}^{c_-+i\infty} s^{-m-2} \tan(\pi s) f(s) t^{-s}
    ds, \\ \bar\rho_+(t) &= \frac{-1}{2\pi i}
    \int_{c_+-i\infty}^{c_++i\infty} s^{-m-2} \tan(\pi s) f(s) t^{-s}
    ds.
  \end{align*}
  Then
  \begin{equation}\label{eq:cauchy-result}
    \bar\rho_+(t) = \frac{2^{m+2}}\pi f(\tfrac12) t^{-1/2} +
    \bar\rho_-(t)
  \end{equation}
  for all $t\in\mathbb R_+$.
\end{lemma}
\begin{proof}
  The integrands in $\rho_+,\rho_-$ are holomorphic in
  $S(a,b)\setminus\{1/2\}$ since $f$ is holomorphic in $S(a,b)$. The
  estimates for the tangent function of
  Lemma~\ref{lem:tangent-estimates} imply that
  \[
  \abs{\tan(\pi s)} \leq C_{c_+}
  \]
  when $\Re s = c_+$. This is because $c_+$ is fixed and away from
  half-integers. This and the estimate for $f$ in the assumptions give
  \begin{equation}\label{eq:inversion-growth-estimate}
    \abs{s^{-m-2}\tan(\pi s)f(s)} \leq K \abs{s}^{-2}
  \end{equation}
  for $\Re s=c_+$. Since $\abs{s}^{-2}$ is integrable on $\{c_+ +
  it\mid t\in\mathbb R\}$ we get
  \begin{equation}\label{eq:rhobar-as-a-limit}
    \bar\rho_+(t) = \lim_{M\to\infty}\frac{-1}{2\pi i}
    \int_{c_+-iM}^{c_++iM} s^{-m-2} \tan(\pi s) f(s) t^{-s} ds
  \end{equation}
  for each $t\in\mathbb R_+$.
  
  Define the following points and paths
  \begin{equation}\label{eq:points-segments}
    \begin{cases}
      P_{+-} = c_+ - iM\\ P_{++} = c_+ + iM\\ P_{-+} = c_- +
      iM\\ P_{--} = c_- - iM
    \end{cases}
    \qquad
    \begin{cases}
      \gamma_{+-}(r) = (1-r)P_{+-} + rP_{++}\\ \gamma_{++}(r) =
      (1-r)P_{++} + rP_{-+}\\ \gamma_{-+}(r) = (1-r)P_{-+} +
      rP_{--}\\ \gamma_{--}(r) = (1-r)P_{--} + rP_{+-}
    \end{cases}
  \end{equation}
  which form a counterclockwise rectangle with the point $s=1/2$ in
  the interior of the loop. The integrand in
  (\ref{eq:rhobar-as-a-limit}) is holormorphic in a neighbourhood of
  this rectangle as long as the neighbourhood is small enough to not
  reach $s=1/2$. For any $t\in\mathbb R_+$ denote the integrand by
  \begin{equation}\label{eq:holom1}
    I_t(s)=s^{-m-2}\tan(\pi s)f(s) t^{-s}, \qquad
    I_t:S(a,b)\setminus\{1/2\}\to\C
  \end{equation}
  to save space.

  By Cauchy's residue theorem
  \begin{equation}\label{eq:residue-eq}
    \frac{-1}{2\pi i}\left( \int_{\gamma_{+-}}+ \int_{\gamma_{++}}+
    \int_{\gamma_{-+}}+ \int_{\gamma_{--}} \right) I_t(s) ds = -
    \operatorname{Res}(I_t,1/2).
  \end{equation}
  Let us calcuate the residue at $s=1/2$. The factors of $I_t$ are
  holomorphic around $s=1/2$ except for $\tan(\pi s)$, whose residue
  is given by Lemma~\ref{lem:tan-residue}. We have
  \begin{align}
    \operatorname{Res}(I_t,\tfrac12) &= \left(\frac12\right)^{-m-2}
    f(\tfrac12) t^{-1/2} \operatorname{Res}(\tan(\pi s),\tfrac12)
    \notag\\ &= -\frac{2^{m+2}}{\pi} f(\tfrac12)
    t^{-1/2}. \label{eq:holom1-residue}
  \end{align}

  \smallskip
  Next, let's investigate what happens when we let $M\to\infty$
  again. For the horizontal segments recall the horizontal estimate
  for the tangent in Lemma~\ref{lem:tangent-estimates}. It implies
  that $\abs{\tan(\pi s)} \leq 2$ when $\abs{\Im(s)}>1/\pi$. The
  estimate for $f$ in the assumptions give a uniform bound for
  $\abs{f(s)/s^m}$. Furthermore, $\abs{t^{-s}} = t^{-\Re(s)} \leq
  t^{-\alpha}$ when $\Re(s)> \alpha$. This value is independent of
  $M$. Lastly, on $\gamma_{++}$ and $\gamma_{--}$ we have
  $\abs{s^{-2}}\leq M^{-2}$, and the lengths of these paths are both
  $c_+-c_-$. Summarising, on $\gamma_{++}$ and $\gamma_{--}$ we have
  \[
  \abs{I_t(s)} \leq C t^{-a} M^{-2},
  \]
  so the integrals over these horizontal paths vanish as
  $M\to+\infty$.

  The integral over $\gamma_{+-}$ multiplied by the constant in front
  of it in (\ref{eq:residue-eq}) equals $\bar\rho_+(t)$, as we saw
  above in (\ref{eq:rhobar-as-a-limit}) when we passed the integral
  limits to infinity. Lastly, just as at the beginning of this proof,
  we can let $M\to\infty$ in the integral over $\gamma_{-+}$, and get
  $-\bar\rho_-(t)$. The claim follows.
\end{proof}

We have all the ingredients to prove Theorem~\ref{thm:no-singularity}
and Theorem~\ref{thm:with-singularity}.

\begin{proof}[Proof of Theorem~\ref{thm:no-singularity}]
  Existence is given by Lemma~\ref{lem:explicit-solution}. Uniqueness
  follows from Lemma~\ref{lem:uniqueness-in-part}.
\end{proof}

\begin{proof}[Proof of Theorem~\ref{thm:with-singularity}]
  The existence and non-existence follow from
  Lemma~\ref{lem:explicit-solution}. Uniqueness is given by
  Lemma~\ref{lem:uniqueness-in-part}. All that's left to prove is the
  identity (\ref{eq:solution-difference}). The existence lemma gives
  us formulas for $\rho_-$ and $\rho_+$ in the form of
  (\ref{eq:rho-m-formula}) and (\ref{eq:rho-p-formula}). These are
  just $(-td/dt)^m$ applied to the integrals in
  Lemma~\ref{lem:cauchy-integral} with $f(s)=\M[e](s)$. Thus
  \[
  \rho_+(t) - \rho_-(t) = \frac{2^{m+2}}\pi \M[e](1/2) \left( -t
  \frac{d}{dt} \right)^m \frac1{\sqrt{t}}.
  \]
  But $t^{-1/2}$ is an eigenfunction of $(-t d/dt)$, since
  \[
  (-t d/dt) t^{-1/2} = -t\cdot(-1/2) t^{-1/2-1} = 2^{-1} t^{-1/2}.
  \]
  Hence $(-t d/dt)^m t^{-1/2} = 2^{-m} t^{-1/2}$ and the resul
  follows.
\end{proof}

\section*{Acknowledgements}
All three authors' work was supported by the Estonian Research
Council’s Grant PRG~832. In addition EB's work on this paper was
partly the Academy of Finland (projects 312124 and 336787).

\end{document}